\documentclass[11pt]{amsart}
\usepackage{amsmath,amssymb,amscd}
\usepackage[matrix,arrow]{xy}
\entrymodifiers={+!!<0pt,\fontdimen22\textfont2>}
\usepackage[a4paper,text={140mm,240mm},centering,headsep=5mm,footskip=10mm]{geometry}

\begin{document}

\newtheorem{theo}{Theorem}[section]
\newtheorem{lem}[theo]{Lemma}
\newtheorem{prop}[theo]{Proposition}
\newtheorem{conj}[theo]{Conjecture}
\newtheorem{coro}[theo]{Corollary}
\newtheorem{sublem}[theo]{Sublemma}
\newtheorem*{IntroTheo}{Theorem}

\newcommand{\lr}{\longrightarrow}
\newcommand{\p}{\pi^{ab}_1}
\newcommand{\coker}{\mathrm{coker}}
\newcommand{\im}{\mathrm{im}}
\newcommand{\chara}{\mathrm{char}}
\newcommand{\Z}{\mathbb{Z}}
\newcommand{\Gal}{\mathrm{Gal}}
\newcommand{\Aut}{\mathrm{Aut}}
\newcommand{\DG}{\mathrm{DG}}
\newcommand{\ord}{\mathrm{ord}}

\theoremstyle{definition}
\newtheorem{defi}[theo]{Definition}

\theoremstyle{remark}
\newtheorem{rem}[theo]{Remark}
\newtheorem{example}[theo]{Example}

\title{Higher class field theory and the connected component}
\author{Moritz Kerz}
\date{25.11.2010}

\thanks{The author is supported by the { DFG-Emmy Noether programme}}
\maketitle

\begin{abstract}
In this note we present a new self-contained approach to the class field theory of arithmetic schemes in the sense of Wiesend. Along the way we prove new results on space filling curves on arithmetic schemes and on the class field theory of local rings. We show how one can deduce the more classical version of higher global class field theory due to Kato and Saito from Wiesend's version. One of our new results says that the connected component of the identity element in Wiesend's class group is divisible if some obstruction is absent.
\end{abstract}

\section*{Introduction}

\noindent The main aim of higher global class field theory is to determine the abelian fundamental group $\p (X)$ of
a regular arithmetic scheme $X$, i.e.~of a connected regular separated scheme flat and of finite type over $\Z$, in terms of an arithmetically defined
class group $C(X)$. In case $\dim(X)=1$ this is classically done in terms of restricted idele groups as follows:

Let $K$ be a number field and $S$ a finite set of places of $K$ containing all infinite places. Let $X$ be the complement of $S$ in $Spec(\mathcal{O}_K)$ as an open subscheme. Define the class group of $X$ to be $$C(X)=\mathrm{coker} [ K^\times \lr \bigoplus_{x\in |X|} \mathbb{Z} \oplus \bigoplus_{v\in S} K^\times_v    ]$$
with the quotient topology of the direct sum topology.  Here $K_v$ is the completion of $K$ at $v$ and $|X|$ is the set of closed points of $X$.
During the first half of the 20th century it was shown that there exists a surjective reciprocity homomorphism $\rho:C(X)\to \p(X)$ whose kernel
is the maximal divisible subgroup of $C(X)$ and which induces a bijection between the open subgroups of $C(X)$ and the open subgroups of $\p(X)$.
In our setting it is this fundamental theorem that one wants to generalize to higher dimensional schemes $X$.
 
For $\dim(X)>1$ a solution to this problem was suggested by Parshin \cite{Par}
and completed by Kato and Saito in~\cite{KS}. Roughly, their solution involves higher Milnor $K$-groups of higher local fields in the definition of the class group as analogs of the multiplicative group of a one-dimensional local field. 
A different approach to higher class field theory was started in~\cite{SchSp} and built into a self-contained approach by Wiesend in~\cite{W4} 
and~\cite{Wclass}.
A completed presentation of Wiesend's ideas was given in~\cite{KeSch}. 

For arithmetic schemes the fundamental result in this latter setting of higher class field theory -- which is the topic of this note -- can be stated as follows:
First, for an arithmetic scheme $X$ one defines a topological group $C(X)$, the class group, together with a
continuous homomorphism $\rho:C(X)\to \p(X)$, the reciprocity map. This is done in Section~4. Sections~5 to~9 are concerned with 
the proof of the fundamental theorem:

\begin{IntroTheo}
For a regular arithmetic scheme $X$ the sequence
\[
0\lr C(X)^0 \lr C(X) \stackrel{\rho}{\lr} \p (X) \lr 0
\]
is a topological short exact sequence. Here $C(X)^0$ is the connected component of the identity in $C(X)$.
Furthermore, if $X$ is proper over some open subscheme of $Spec(\mathcal{O}_K)$ for some number field $K$ then the group $C(X)^0$ is the maximal divisible subgroup of $C(X)$.
\end{IntroTheo}

Our principal aim for this note was to give a new, short and direct proof of this fundamental theorem of Wiesend's higher class field theory without using the notion of covering data resp.~covering problems, which is used in \cite{Wclass} and \cite{KeSch}. 
One of the major differences of our approach is that we prove the isomorphism theorem first and use it in the proof of the existence theorem.
The statement on the divisibility of the connected component is new, see Section~5. Furthermore,
we show in Section~10 that this fundamental theorem implies the results of Kato and Saito on arithmetic schemes as presented in~\cite[Theorem 6.1]{Raskind}.
One should remark that in the case of varieties over finite fields we can describle at present only the tame part of the abelian fundamental group using the approach 
of Wiesend. Here the approach of Kato and Saito seems to be indispensable in order to describe the wild part.

It is suggested to the reader who wants to gain an overview of higher class field theory to skip the first three sections -- which are quite technical -- for the first reading and start with Section~4. 

Part of this note grew out of a seminar on Wiesend's work which was held at the University of Regensburg. I would like to thank the participants of this seminar and especially Alexander Schmidt for many interesting discussions on higher class field theory and his comments on preliminary versions of this note. I thank Uwe Jannsen for his constant encouragement. The referee suggested helpful improvements, in particular for Example~\ref{examcon}.
\section{Background Material}

\subsection{Algebraic Geometry}

\subsubsection*{General definitions}

\begin{defi}
An arithmetic scheme is an integral separated normal scheme flat and of finite type over $\mathbb{Z}$.
\end{defi}
\begin{defi}
If $X$ is a scheme we call a closed subscheme $C\to X$ a curve if $C$ is
integral and one-dimensional. 
\end{defi}

For a scheme $X$ we denote by $|X|$ the set of closed points of $X$.

\subsubsection*{Elementary fibrations}
Usually proofs in higher class field theory use induction over the Krull dimension. A basic
ingredient will be Artin's elementary fibrations.
Let $X$ be an arithmetic scheme with $\dim(X)>1$. The next proposition shows that
\'etale locally around the generic point $X$ can be fibred into smooth curves over a regular base.

\begin{prop}\label{ElFib}
There exists a nontrivial arithmetic scheme $X'$ and an \'etale morphism $X'\to X$ such that $X'$ can be fibred as
follows: There exists an open immersion $X'\to \bar{X}'$ and a smooth projective morphism $\bar{\pi}: \bar{X}' \to W$
such that:\\
(i) $W$ is regular of dimension $\dim(X)-1$.\\
(ii) $\bar{X}'-X'$ with the reduced subscheme structure is isomorphic to a direct sum of copies of $W$.\\
(iii) There exists a section $s:W\to X'$ of $\bar{\pi}$.
\end{prop} 

\begin{proof} This is a simple corollary to \cite[XI 3.3]{SGA4-3}. \end{proof}

\subsubsection*{Fundamental group}
Let $\mathfrak{Gr}$ be the category whose objects are profinite groups and whose morphisms are continuous group homomorphisms modulo inner
automorphism, i.e.~ two group homomorphisms $f,g:A\to B$ give the same morphism in the category $\mathfrak{Gr}$ if there exists $b\in B$ with 
$b f b^{-1} =g$. Abelianization is a functor from the category $\mathfrak{Gr}$ to the category of 
abelian profinite groups. Let $A$ be a set and $B$ be a group. An equivalence class of maps from $A$ to $B$ up to inner automorphisms is called a subset up to inner automorphism if it corresponds to an injective map.

Grothendieck's fundamental group is a  covariant functor, denoted $\pi_1$, from the category of connected noetherian
schemes to $\mathfrak{Gr}$. The abelian fundamental group functor $\p$ is the composition of $\pi_1$ with the abelianization
functor described above.

For a connected noetherian scheme $X$ the open normal subgroups in $\pi_1(X)$ 
correspond bijectively to the isomorphism classes of (connected) Galois coverings of
$X$. A connected \'etale covering $Y\to X$ is called a Galois covering if $\# \Aut(Y/X)=\deg(Y/X)$.
The Galois covering corresponding to $V\le \pi_1(X)$ has Galois group $G=\pi_1(X)/V$.
 If $f:X'\to X$ is a
morphism of connected normal noetherian schemes  and $Y\to X$ is a Galois covering corresponding to
$V\le \pi_1(X)$ the pullback $Y_{X'}=X' \times_X Y$ is a disjoint union of Galois coverings each of which
has Galois group isomorphic to $G'=\pi_1(X') /f_*^{-1}(V)$.

For an arithmetic scheme $X$ and a closed point $x\in |X|$ we will denote the Frobenius at $x$ by $Frob_x$, which is an element of $\pi_1(X)$ up to conjugation. If $Y\to X$ is a Galois covering we call the subgroup up to conjugation of $\Gal(Y/X)$
generated by $Frob_x$ the decomposition group at the point $x$, denoted $\DG_{Y/X}(x)$.

For further reference recall:
\begin{lem}\label{fcov}
If $f:Y\to X$ is an \'etale morphism of arithmetic schemes there exists a dense open subscheme $U\subset X$ such that the
induced morphism $f^{-1}(U)\to U$ is finite.
\end{lem}

\subsubsection*{Etale cohomology}

Etale cohomology of a connected noetherian scheme $X$ is related to the fundamental group by 
$$Hom_{cont}(\pi_1(X) , \mathbb{Q} / \mathbb{Z} ) =  H^1(X,\mathbb{Q}/\mathbb{Z})\; . $$

We will need a base change result for relative curves.
Let $\bar{\pi}:\bar{X} \to W$ be a smooth proper relative curve with geometrically connected fibres, let $W$ a regular arithmetic scheme. 
Let $i:X\to \bar{X}$ be an open immersion such that
$\bar{X} -X$ is isomorphic to a direct sum of copies of $W$ and such that there exists a section $s:W\to X$.
Set $\pi = \bar{\pi} \circ i$.

\begin{prop}\label{EFiniteness}
The base change homomorphism
\[
(R^1 \pi_* (\mathbb{Z}/m) )_{\bar{w}}\tilde{\lr} H^1(X\times_W \bar{w} ,\mathbb{Z}/m )
\]
for a geometric point $\bar{w}$ of $W$ is an isomorphism for all $m\in \mathbb{N}$ which are invertible in $H^0(W,\mathcal{O}_W)$.
Moreover $R^1 \pi_* (\mathbb{Z}/m)$ is locally constant.
\end{prop}

\begin{proof} Purity \cite[Corollary 5.3]{Milne} implies that there is a commutative diagram

$$ \xymatrix @-1pc { 
0 \ar[r]  &  R^1 \bar{\pi}_* ( \mathbb{Z}/m )_{\bar{w}}  \ar[r] \ar[d] &  R^1 \pi_* (\mathbb{Z}/m)_{\bar{w}} \ar[r] \ar[d]  &     
 \bigoplus\limits_{\pi_0(\bar{X}-X)} 
\mathbb{Z}/m(-1) \ar[r] \ar[d] & \mathbb{Z}/m(-1) \ar[d] \\
 0 \ar[r] &  H^1(\bar{X}\times_W \bar{w} ,\mathbb{Z}/m ) \ar[r] &  H^1(X\times_W \bar{w} ,\mathbb{Z}/m ) \ar[r] & \bigoplus\limits_{\pi_0(\bar{X}-X)} 
\mathbb{Z}/m(-1) \ar[r] & \mathbb{Z}/m(-1) } 
 $$ 
with exact rows.
So the result follows from the smooth proper base change and finiteness theorem \cite[Corollary 4.2]{Milne} and the five lemma.
\end{proof}
 
\subsubsection*{Compactification of curves}

Let $C$ be a reduced scheme of dimension one, separated and of finite type over $\mathbb{Z}$.

\begin{prop}\label{compcurve}
There exists a scheme $\bar{C}$ which is proper over $\Z$ and a dense open immersion $C\subset \bar{C}$ with the following property:
Every morphism from $C$ to a scheme $X$ which is proper over $\Z$ factors uniquely through $\bar{C}$. This clearly determines $\bar C$ up to unique
isomorphism.
\[
\xymatrix{
C \ar[r] \ar[d] & \bar C \ar@{-->}[dl]^(.4){\exists !}\\
X &
}
\]
\end{prop}

\begin{proof}
By a theorem of Nagata, see~\cite{Luet}, there exists a reduced compactification $C\subset \bar{C}'$. Let $\bar{C}$ be the desingularization
of $\bar{C}'$ at the points $\bar{C}'\backslash C$. An application of the valuative criterion of properness finishes the proof.
\end{proof}

\subsection{Topological Groups}\mbox{}\\
Let $(G,e)$ be a topological group. For further reference recall:
\begin{lem}\label{HausGr}
$\{e\}$ is closed in $G$ if and only if $G$ is Hausdorff.
\end{lem}

%begin{proof} If $G$ is Hausdorff it is well known that all points are closed. 
%Let $d:G\times G\to G $ be the continuous map $(g_1,g_2)\mapsto g_1 - g_2$. Then $\Delta = d^{-1} (\{e\})$ is the diagonal
%in $X\times X$. So if $\{e\}$ is closed $\Delta$ is closed too. But this is equivalent to $G$ being Hausdorff.
%\end{proof}

By $G^0$ we denote the maximal connected subset of $G$ containing $e$. It is well known that $G^0$ is a closed
subgroup.

\subsection{One-dimensional class field theory.}\mbox{}\\
Let $K$ be a global field and $S$ a finite set of places of $K$ containing all infinite places. Let $X$ be
the open subscheme of $Spec(\mathcal{O}_K)$ whose closed points are exactly the closed points of $Spec(\mathcal{O}_K)$
not in $S$. Denote by $$C(X)=\mathrm{coker} [ K^\times \lr \bigoplus_{x\in |X|} \mathbb{Z} \oplus \bigoplus_{v\in S} K^\times_v    ]$$ the (idele) class group
of $X$. Here $K_v$ is the completion of $K$ at $v$. $C(X)$ is a locally compact Hausdorff group and covariant functorial in $X$. 
There exists a canonical continuous homomorphism $$\rho: C(X) \lr \p (X)$$
called the reciprocity map. If $p=\chara(K)>0$ we have the natural homomorphism $ C(X) \to \mathbb{Z} $ induced by
$X\to Spec(\mathbb{F}_p)$. 
The main theorem of class field theory, as proved for example in \cite[Chapter 8]{ArtinTate}, reads now:

\begin{prop}\label{OneClass}
If $\chara(K)=0$ the sequence
\[
0 \lr C(X)^0 \lr C(X) \lr \p(X) \lr 0
\]
is a topological short exact sequence.\\
If $p=\chara(K)>0$ the homomorphism 
\[
\ker[C(X)\to \mathbb{Z} ] \lr \ker[ \p (X) \to \p ( \mathbb{F}_p ) = \hat{\mathbb{Z}} ]
\]
is a topological isomorphism.
\end{prop}

Let now $X$ be as above and $\phi:Y\to X$ be a Galois covering, $G=\Gal(Y/X)$. 
\begin{coro}[Isomorphism Theorem]\label{IsoOne}
The reciprocity map induces an isomorphism
\[
C(X)/\phi_* C(Y)  \tilde{\lr} G^{ab}\; .
\]
\end{coro}

\begin{proof}
One checks that $\phi_* C(Y)$ is an open subgroup of $C(X)$, in particular we have $C(X)^0\subset \phi_* C(Y)$. Now a simple diagram chase, using
Proposition~\ref{OneClass} for $X$ and $Y$, proves the corollary.
\end{proof}

For further reference we finally recall the weak approximation lemma. Let $F$ be a field and $|\cdot |_1, \ldots  , |\cdot |_n$ inequivalent
valuations on $F$.

\begin{lem}[Weak approximation lemma]\label{ElApprox}
Given $a_1,\ldots , a_n\in F$ and $\epsilon >0$ there exists $a\in F$ with
$|a-a_i|_i<\epsilon $ for all $1\le i \le n$.
\end{lem} 

\subsection{Chebotarev density}\mbox{}\\ 
Let $X$ be an arithmetic scheme, $d=\dim(X)$. One can show that
$
\sum_{x\in |X|} \frac{1}{\mathbf{N}(x)^{s}}
$
converges for $s>d$.
Here for a closed point $x$ in $X$ we denote by $\mathbf{N}(x)$ the number of elements in $\mathbf{k}(x)$.

Given a subset $M\subset |X|$ we call
\[
D(M)=\lim_{s\to d+0} \left( \sum_{x\in M} \frac{1}{\mathbf{N}(x)^s}  \right) / \log(\frac{1}{s-d})
\] 
the density of $M$ if the limit exists.
Let $Y \to X$ be a Galois covering with $G=\Gal(Y/X)$ 

\begin{prop}[Generalized Chebotarev]\label{Chebotarev}
If $R\subset G$ is stable under conjugation and $M= \{ x\in |X|\: |\: Frob_x\in R \}$ the density
$D(M)$ is well defined and
\[
D(M)= \frac{\# R}{\# G}\; .
\]
\end{prop}

For a proof see for example \cite{Serre}.

\subsection{Katz-Lang finiteness}\mbox{}\\
Let $f:X\to S$ be a smooth surjective morphism of arithmetic schemes. Assume the geometric generic
fibre of $f$ is connected. Then Katz and Lang \cite{KatzLang} prove the following geometric finiteness result:
\begin{prop}[Katz-Lang]\label{Katz-Lang}
The kernel of $f_*:\p (X) \to \p (S)$ is finite.
\end{prop}

\section{Bloch's approximation lemma}

\noindent An essential method in higher global class field theory is to reduce things to the one-dimensional case by the following
procedure (for a simple incidence of the method see the proof of the Isomorphism Theorem in Section~\ref{SecIso}):
Given an arithmetic scheme $X$ and a finite set of closed points on the scheme find a `good' curve on $X$ which contains the
given points.
The most general conjecture in this direction would be:

\begin{conj}[Space filling curves on arithmetic schemes]
Let $X$ be a regular quasi-projective arithmetic scheme, i.e.~$X$ is a subscheme of $\mathbb{P}^n_{\Z}$, and let $S\subset |X|$ be a finite set of closed points. Then there exists a regular curve $C$ on $X$ with $S\subset |C|$.
\end{conj}
A conditional result for $S=\emptyset$ has been proven in~\cite{Poonen}; the idea is to use hyperplane sections.
In higher class field theory the following very weak form of the conjecture is sufficient. The next proposition strengthens Bloch's approximation lemma \cite[Lemma 3.3]{Bloch}, \cite[Lemma 6.21]{Raskind}, but whereas Bloch uses Hilbert irreducibility we use a Bertini theorem over finite fields and some classical algebraic geometry.

Let $F$ be a number field and $\mathcal{O}$ its ring of integers.

\begin{prop}[Bloch approximation]\label{Approx}
Let $X/\mathcal{O}$ be a smooth quasi-projective arithmetic scheme and let $Y\to X$ be a
Galois covering.  Let $x_i$ ($1\le i\le n$) be a finite set of closed points of $X$. Then there exists a curve $C\subset X$ such that\\
(i) The points $x_i$ are regular points of the curve $C$.\\
(ii) The scheme $Y\times_X C$ is irreducible.
\end{prop}

Before we give the proof, we have to recall a Bertini type theorem over finite fields.
Let $X$ be a quasi-projective subscheme of $\mathbb{P}^n_k$ where $k$ is some finite field. For a nonvanishing section $f\in H^0(\mathbb{P}^n_k ,\mathcal{O}(d))$
we denote by $H_f$ the corresponding hypersurface.

\begin{prop}[Gabber-Poonen]\label{GP}
Assume $X$ is smooth and $x_i$ ($1\le i \le m$) is some finite family of closed points of $X$. For  $d\gg 0$ there exists
$f\in H^0(\mathbb{P}^n_k ,\mathcal{O}(d))$, such that $X$ and $H_f$ intersect properly, $X\cap H_f$ is smooth and $x_i\in H_f$ for $1\le i \le m$.
\end{prop}

\begin{proof} Gabber \cite[Corollary 1.6]{Gabber} proves the proposition under the assumption that $\chara (k) | d$. Poonen \cite[Theorem 1.2]{Poonen} proves it for large arbitrary
$d$. \end{proof}
Below we have to use the following standard connectivity fact about hypersurfaces, which is shown for example in \cite[Corollary III.7.9]{Hart}:

\begin{prop}\label{con}
If $H\subset \mathbb{P}^n_k$ is a hypersurface and if $X \subset \mathbb{P}^n_k$ is a closed subscheme with $\dim(X)\ge 2$, which is geometrically irreducible and 
smooth over $k$, then the intersection $H\cap X$ is geometrically connected.
\end{prop}

\begin{proof}[Proof of Bloch approximation.] In the first part of the proof we use induction on $\dim(X)$ to reduce to the case $\dim(X)=2$. In the second part we handle the case $\dim(X)=2$. 

{\bf 1st part:} Assume $\dim(X)>2$ and the theorem is known for two-dimensional schemes -- this case will be validated in the second part. 
´According to Proposition \ref{Chebotarev} we can find closed points $x_i\in X$ ($n< i\le m$) such that each
conjugacy class in $\Gal(Y/X)$ contains at least one Frobenius $Frob_{x_i}$ for some $i$. \\[2mm]
{\bf Claim} ($\dim>2$). {There exists a curve $C$ on $X$ which contains $x_i$ ($1\le i\le m$) as regular points.}
\begin{proof} 
We prove the claim by induction on $\dim(X)\ge 2$. The case $\dim(X)=2$ is shown in the 2nd part below.  Let $Z\subset |Spec(\mathcal{O})|$ be the image of the set of points $\{x_i| 1\le i\le m\}$
and denote $\eta$ the generic point of $Spec(\mathcal{O})$.
 Write $X$ as a subscheme of $\mathbb{P}^N_\mathcal{O}$. $\bar{X}$ will denote the closure of $X$ in $\mathbb{P}^N_\mathcal{O}$. 
Using Hironaka's resolution of singularities at the generic fibre we can assume
 without restriction that $\bar{X}_\eta$ is smooth over $F$. After replacing $F$ by the algebraic closure of
$F$ in $X_\eta$ we can assume that $\bar{X}_\eta$ is geometrically irreducible over $F$. 
Then we can find a prime ideal $p$ in $\mathcal{O}$ distinct from the primes
in $Z$ such that $\bar{X}_{\mathcal{O}_{p}}$ is smooth over $\mathcal{O}_{p}$ and such that $\bar{X}_p=\bar{X}\otimes \mathbf{k}(p)$ is irreducible. The latter because of
Zariski's connectedness theorem~\cite[III.11.3]{Hart}.
Let $A$ be the semi-local ring corresponding to the finite set of points $Z\cup \{ p\}$ of $Spec(\mathcal{O})$ and $I$ its Jacobson radical.
Proposition \ref{GP} says, that for $d\gg 0$ there exists a global section of  $\mathcal{O}_{\mathbb{P}^N_{A/I}}(d)$ 
which induces a hypersurface $H$ of $\mathbb{P}^N_{A/I}$ whose  intersection with $\bar{X}\otimes_\mathcal{O} A/I $ is proper, contains the points $x_i$ 
as smooth points and such that $H \cap \bar{X}_{p}$ is smooth.
Let $im:\mathbb{P}^N_{A/I} \to \mathbb{P}^N_{A}$ be the closed embedding.
The homomorphism
\[
im^*:H^0(\mathbb{P}^N_A ,  \mathcal{O}(d) ) \lr H^0(\mathbb{P}^N_{A/I} ,  \mathcal{O}(d) )
\]
is surjective due to the Chinese remainder theorem.
This enables us to lift $H$ to a hypersurface $H_A$ in $\mathbb{P}^N_A$. Observe that $H\cap \bar{X}_p$ is smooth and connected, the latter by Proposition \ref{con}.
This implies that $H_A \cap \bar{X}$ is irreducible.
Let $X'$ be a smooth connected subscheme of $X$ such that $X'\otimes A= H_A\cap X$. 
Then $X'$ is an integral smooth quasi-projective arithmetic scheme over $\mathcal{O}$ of dimension $\dim(X)-1$ containing the points
$x_i$, so that we can apply an induction to reduce to the case $\dim(X)=2$ which is shown below in the 2nd part of the proof.
\end{proof}

Now assume $Y_C=Y\times_X C$ was not irreducible, where $C$ is as in the claim. This would mean that
the composite $\pi_1(C) \to \pi_1(X) \to \Gal(Y/X)$ was not surjective. But as its image contains Frobenius elements in every conjugacy class of $\Gal(Y/X)$, this would give a 
contradiction because of Lemma \ref{GroupGen}.

{\bf 2nd part:} Now we assume $\dim(X)=2$ and consider an embedding $X\to \mathbb{P}^N_\mathcal{O}$.
Let again $Z\subset |Spec(\mathcal{O})|$ be the image of the set of points $\{x_i| 1\le i\le m\}$. Let $(X_0,\ldots ,X_N)$ be homogeneous
coordinates for $\mathbb{P}^N_\mathcal{O}$.
 After performing an $n$-uple
embedding and a linear change of variables we can without restriction assume that $x_i \notin H_{X_0}$ for $1\le i\le m$. Let $A$ be the semi-local
ring with Jacobson radical $I$ corresponding to the finite set of points $Z$ of $Spec(\mathcal{O})$.
Then by Proposition \ref{GP} for $d\gg 0$ we find a section $f'\in H^0(\mathbb{P}^N_{A/I},\mathcal{O}(d))$ such that the hypersurface $H_{f'}$ 
of $\mathbb{P}^N_{A/I}$ has proper intersection with $\bar{X} \otimes_\mathcal{O} A/I$ and the intersection contains the points $x_i$ as smooth points. 
Let $f\in H^0(\mathbb{P}^N_{A},\mathcal{O}(d))$ be a preimage of 
$f$ under the natural map $im_*$ described above.

Now consider the rational map $\phi:\mathbb{P}_A^N \to \mathbb{P}^1_A $ induced by $(X_0^d,f)$. After shrinking  $X$ we can assume that $\phi$ induces a
morphism $\phi|_X:X \to \mathbb{P}^1_\mathcal{O}$.
It is \'etale at the points $x_i$ for $1\le i\le n$. In fact it is enough to show the latter fibrewise for the fibres 
over $Z$ where it follows from the choice of $f'\in H^0(\mathbb{P}^N_{A/I},\mathcal{O}(d))$. Now by further shrinking $X$ around the points $\{x_i| 1\le i\le m\}$
 we can assume that $\phi|_X$ is \'etale.

According to Lemma \ref{fcov} we can find an open subscheme $U\subset \mathbb{P}^1_\mathcal{O}$ such that
$$(\phi|_{X})^{-1}(U) \to U$$ is an \'etale covering. Restricting $Y$ to $\phi|_X^{-1}(U)$ we obtain an \'etale covering of $U$ and denote by $Y_U\to U$ its Galois closure. Choose a finite set
 of closed points $x_j\in |U|$ ($n< j\le m$) such that each conjugacy class in $\Gal(Y_U /U)$ contains one of the Frobenius elements $Frob_{x_j}$ ($n<j \le m$).\\[2mm]
{\bf Claim} ($\dim=2$) {There exists a curve $C$ on $\mathbb{P}^1_\mathcal{O}$ which contains $\phi(x_i)$ ($1\le i\le n$) and
$x_j$ ($n<j\le m$) as regular points.}

\medskip

Let $C$ be as in the claim. The curve $(\phi|_X)^{-1}(C)$ is the curve we were looking for. 
In fact it is irreducible because $Y_U \times C$ is irreducible. The latter because otherwise $\pi_1(C\cap U) \to \pi_1(U) \to \Gal(Y_U/U)$ would not
be surjective in contradiction to Lemma \ref{GroupGen}, as the image contains Frobenius elements $Frob_{x_j}$ ($n<j\le m$) in every conjugacy class.
\end{proof}

\begin{lem}\label{GroupGen}
Let $G$ be a finite group and $D\subset G$ a subset which contains elements from every conjugacy class. Then $D$ generates $G$.
\end{lem}

\begin{proof} Let $H$ be the subgroup generated by $D$ and $i=[G:H]$. By assumption we cleary have
\[
G = \bigcup_{x\in G/H} x H x^{-1} \; .
\] 
If $i>1$ the union is not disjoint and counting the elements on both sides would give $\ord(G) < i \cdot \ord(H)$, which is a contradiction.
\end{proof}

%----------------------
%---------------------

\section{Splitting properties}

\noindent We will need two different splitting results, one global which is based on Chebotarev density and one local which
is based on class field theory of henselian local rings.

We begin with the global result which is fairly standard.

\begin{prop}[Global splitting]\label{GlobalSplit}
Let $X$ be an arithmetic scheme and $\phi:Y\to X$ be a connected \'etale covering which splits completely over all closed points of $X$.
Then $\phi$ is an isomorphism.
\end{prop} 
\begin{proof} Without restriction we can assume $\phi$ to be a Galois covering with Galois group $G$. Because
every closed point in $X$ splits completely, all the Frobenius elements are trivial in $G$. But according to Theorem \ref{Chebotarev} the Frobenius elements generate $G$, so $G$ is
trivial. \end{proof}

Now we come to the local result. Let $A$ be an excellent regular henselian local  ring of dimension $d$ with finite residue field. 
Let $X$ be a dense open subscheme of $Spec(A)$. Set $D=Spec(A)-X$. 
The next proposition is due to Saito \cite{Saito} for $d=2$. A geometric 
proof can be found in \cite{KeSch2}. Here we will give a new proof generalizing the
work of Saito.

\begin{prop}[Local splitting]\label{LocalSplit}
Let $\phi:Y\to X$ be an abelian Galois covering which splits completely over all closed points of $X$. Then
$\phi$ is an isomorphism.
\end{prop}

\begin{proof} The proof uses class field theory of henselian local rings. Without restriction we can assume $X\ne Spec(A)$. Let $\mathcal{P}$ be the set of Parshin chains on $Spec(A)$ of the form
$P=(p_0, \ldots , p_n)$ ($0\le n\le d$) such that $p_0,\ldots ,p_{n-1} \in D$ and $p_n\in X$.
Let $\mathcal{R}$ be the set of chains $P=(p_0, \ldots , p_n)$ ($0\le n <d$) such that $\dim(p_i)=i$ with
$p_i\in D$
for $0\le i<n$ and $\dim(p_n)= n+1$ with $p_n\in X$. For a chain $P=(p_0,\ldots ,p_n)$ we call $\dim(P)=\dim(p_n)$ the
dimension of $P$.
 Now remember that to every chain $P$ we can associate
a finite product of fields $\mathbf{k}(P)$ by a henselization process, see \cite[Section 1.6]{KS}. Define the idele group to be
\[
I(X)= \bigoplus_{P\in \mathcal{P}} K^M_{\dim(P)}(\mathbf{k}(P))\; .
\] 
We endow $I(X)$ with the following topology: A neighborhood base of the zero element is given by 
the subgroups 
\[
\bigoplus_{P\in \mathcal{P}\atop \dim(P)=d} K^M_{d}(\mathbf{k}(P),m) \le I(X)
\]
where $m\in \mathbb{N}$ and for a discrete valuation ring $(R,I)$ with quotient field $F$ we let 
$K^M_n(F,m)$ be the subgroup of $K^M_n(F)$ generated by symbols $\{1+ I^m , F^\times ,\ldots ,F^\times \}$.
Define the class group to be the obvious quotient
\[
C(X)=  \mathrm{coker} \left[\bigoplus_{P\in \mathcal{R}} K^M_{\dim(P)} (\mathbf{k}(P))   \to I(X)\right]  \; .
\]
Now it follows from a reciprocity result due to Kato \cite[Proposition 7]{Kato} that the natural reciprocity homomorphism
$I(X)\to \p (X)$ factors through a continuous $\rho: C(X) \to \p (X)$. The proposition follows from the next lemma whose
proof we leave to the reader. It is only a slight generalization of the two-dimensional case treated in
\cite{Saito}. The proof uses Zariski-Nagata purity of the branch locus and approximation in Dedekind rings, Lemma \ref{ElApprox}.
\begin{lem}
The map $\rho$ has dense image. The natural map 
\[
h:\bigoplus_{x\in |X|} \mathbf{k}(x)^\times \lr C(X)
\]
has dense image.
\end{lem}
It follows from the lemma that the composite
\[
\bigoplus_{x\in |X|} \mathbf{k}(x)^\times \stackrel{h}{\lr} C(X) \stackrel{\rho}{\lr} \p (X) \lr \Gal(Y/X)
\]
is surjective. But the splitting assumption of the proposition implies that it is the zero homomorphism. \end{proof}

%------------------------------

\section{The Class Group}

\noindent Wiesend introduced a very simple class group for an arithmetic scheme $X$ which is an extension of the Chow group of zero cycles $CH_0(X)$.
For a curve $C$ on $X$ we let $C_\infty$ be $\bar C \backslash C$ together with the archimedean places of $\mathbf{k}(C)$, where $\bar C$
is the compactification of $C$ defined in Proposition~\ref{compcurve}. In other words $C_\infty$ is the set of places of $\mathbf{k}(C)$ which
do not lie over points of $C$. Let $\mathbf{k}(C)_v$ be
the completion of $\mathbf{k}(C)$ with respect to the place $v$.

 %and let $\mathbf{k}(C)^h_v$ be the henselization of $\mathbf{k}(C)$ with respect to $v$,
%which is defined as the algebraic closure of $\mathbf{k}(C)$ in $\mathbf{k}(C)_v$ if $v$ is archimedean. 
%%For our refined class group we need a mixture of these two concepts: Namely, let $\mathbf{k}(C)'_v$ be the completion with respect to $v$ if $\chara(\mathbf{k}(C))=0$ and
%the henselization if $\chara(\mathbf{k}(C))>0$.

\begin{defi}
The idele group of an arithmetic scheme $X$ is defined as the direct sum of topological groups 
\[
I(X)=\bigoplus_{x\in |X|} \mathbb{Z} \oplus \bigoplus_{C, v\in C_\infty} {\mathbf{k}(C)}_v^\times
\]
where $C$ runs over all curves on $X$. The finite idele group $I^f(X)$ is defined in the same way but without the archimedean summands.
\end{defi}

If $\dim(X)>1$ the group $I(X)$ is Hausdorff but not locally-compact. For its connected component of the identity element we have
\[
I(X)^0=\bigoplus \mathbb{C}^\times \oplus \bigoplus \mathbb{R}^\times_+ 
\]
where we sum over all archimedean valuations corresponding to curves on $X$.

The next lemma is essential for our approach to higher class field theory.

\begin{lem}\label{Base}
The open subgroups of $I(X)/I(X)^0$ form a neighborhood base of the zero element.
\end{lem}

\begin{proof} $I(X)$ can be decomposed as $I(X)=I^{\mathrm{a}}(X) \oplus I^{\mathrm{na}}(X) \oplus I^{|X|}(X)$
in an archimedean, non-archimedean and closed point part. Then $I(X)^0$ is a subgroup of $I^\mathrm{a}(X)$ with discrete quotient
group. $I^{|X|}(X)$ is discrete too. So it suffices to show that the open subgroups of $I^{\mathrm{na}}(X)$ form a neighborhood base of the zero element.
The set of curves on an arithmetic scheme is at most countable, so if we assume $\dim(X)>1$ we can write
\[
I^{\mathrm{na}}(X) = \bigoplus_{i\in \mathbb{N} }F^\times_i
\]
for local non-archimedean fields $F_i$ ($i\in \mathbb{N}$). Let $O\subset I^{\mathrm{na}}(X)$ be an open neighborhood of the
zero element. We will successively construct open compact subgroups $U_n$ of $I_n = \oplus_{i\le n} F^\times_i$ contained
in $O$. Suppose we have constructed $U_n$ for some $n$. For each $x\in I_n$ choose an open compact subgroup $O'_x \subset
F^\times_{n+1}$ and an open neighborhood $O_x$ of $x$ in $U_n$ such that $O_x \times O'_x \subset O$. As $U_n$
is compact there exists a finite set $W\subset U_n$ with $\cup_{x\in W} O_x = U_n$. Now set 
$U_{n+1}= U_n \oplus (\cap_{x\in W} O'_x )$. It is now clear that $\cup_{n\in \mathbb{N}} U_n$ is an open subgroup
of $I^{\mathrm{na}}(X)$ contained in $O$.\end{proof}

\begin{defi}
The class group is defined to be
\[
C(X)=   \coker[\bigoplus_{C} \mathbf{k}(C)^\times \to I(X)]
\]
with the quotient topology. The finite variant $C^f(X)$ is defined by the same formula  replacing $I(X)$ by $I^f(X)$.
\end{defi}

\begin{rem}
In~\cite[Example 7.1]{KeSch} it is shown that for $X=\mathbb{P}^1_\mathbb{Z}$ the class group $C(X)$ is not Hausdorff. This might suggest
to replace the class group by its Hausdorff quotient. Nevertheless it often seems difficult to explicitly determine the Hausdorff quotient even for simple arithmetic
schemes, so we do not pursue this approach here. 
\end{rem}

\begin{lem}\label{Predense}
The image of $\bigoplus_{x\in |X|} \mathbb{Z} \to C(X)$ is dense in $C(X)$.
If moreover $X$ is regular and $U\subset X$ is a dense open subscheme the image of $\bigoplus_{x\in |U|} \mathbb{Z} \to C(X)$ is dense in $C(X)$.
\end{lem}

\begin{proof} Let $C\subset C(X)$ be a closed subset containing
the image of $\bigoplus_{x\in |X|} \mathbb{Z} \to C(X)$. We have to show $C=C(X)$. For a curve $D$ on $X$ let $\tilde D$ be the normalization and $i:\tilde D\to X$ the natural morphism. The closed set $i_*^{-1}(C)$ contains
the image of $\bigoplus_{x\in |\tilde D|} \mathbb{Z} \to C(\tilde D)$ so by the weak approximation lemma, Lemma~\ref{ElApprox},
$i_*^{-1}(C) = C(\tilde D)$. This means that $C=C(X)$.

Let now $U$ be a dense open subscheme of a regular $X$. By what has just been shown it is enough to verify that the image of
$\bigoplus_{x\in |U|} \mathbb{Z} \to C(X)$ is dense in the image of $\bigoplus_{x\in |X|} \mathbb{Z} \to C(X)$.  
Fix a closed point $x\in |X|$ and choose a curve $D$ which contains $x$ as a regular point and which meets $U$. Here
we need the regularity of $X$. 
Denote by $1_x$ the element of $C(X)$ which has vanishing summands except at $x$, where it is $1\in \mathbb{Z}$. By the choice of $D$ we have
$1_x\in \im [C(\tilde D) \to C(X) ]$. As the image of $\bigoplus_{x\in |\tilde D\times U|} \mathbb{Z} \to C(\tilde D)$ is dense by Lemma~\ref{ElApprox},
we conclude that $1_x$ lies in the closure of the image of $\bigoplus_{x\in |U|} \mathbb{Z} \to C(X)$.
 \end{proof}

\begin{prop}\label{Intersec}
The intersection of all open subgroups of $C(X)$ is the connected component of the identity $C(X)^0$ in $C(X)$, which is also the closure of the image of $I(X)^0$.
\end{prop}

\begin{proof} Observe that the closure of $\im(I(X)^0)$ in $C(X)$ is contained in $C(X)^0$. So it suffices to show that $\overline{\im(I(X)^0)}\subset C(X)$ is the intersection of all open subgroups of $C(X)$.  It follows from Lemma \ref{HausGr} that $G:=C(X) /\overline{\im(I(X)^0)}$ is a Hausdorff group. So the intersection of all open sets
in $G$ containing $0$ is $\{ 0\}$. We have to show that every open subset $O$ 
of $G$ with $0\in O$ contains an open subgroup. The quotient map $$q:I(X)/I(X)^0\to G$$  is open.
According to Lemma \ref{Base} we can find an open subgroup $U\le I(X)/I(X)^0$ such that $U\subset q^{-1}(O)$. Then
$q(U)$ is the open subgroup of $G$ we are looking for. \end{proof}

\begin{prop}
For a morphism of arithmetic schemes $f:X\to Y$ there exists a unique continuous homomorphism
$f_*:C(X)\to C(Y)$ such that for every closed point $x\in |X|$, $y=f(x)$, the diagram
\[
\xymatrix{
C(x)=\mathbb{Z}  \ar[r]  \ar[d]_{\deg(\mathbf{k}(x)/\mathbf{k}(y))} &  C(X) \ar[d]^{f_*}\\
C(y) = \mathbb{Z}  \ar[r] & C(Y)
}
\]
commutes. The left vertical arrow is multiplication by $\deg(\mathbf{k}(x)/\mathbf{k}(y))$.
\end{prop}

\begin{proof} Uniqueness follows from Lemma \ref{Predense}. Existence can be shown similarly to~\cite[Lemma 7.3]{KeSch} \end{proof}

\begin{prop}[Reciprocity]
There exists a unique continuous homomorphism $\rho: C(X) \to \p (X)$ such that for every closed point
$x\in  |X|$ the diagram
\[
\xymatrix{
C(x) =\mathbb{Z} \ar[r] \ar[d]_{Frob} &  C(X) \ar[d]^{\rho}\\
\p (x) = \hat{\mathbb{Z}}   \ar[r] & \p (X) 
}
\]
commutes.
\end{prop}

\begin{proof} Again uniqueness follows from Lemma \ref{Predense} and existence is shown in analogy to~\cite[Proposition 7.5]{KeSch}. \end{proof}

\begin{coro}
For a morphism of arithmetic schemes $f:X\to Y$ the diagram 
\[
\xymatrix{
C(X) \ar[r]^\rho \ar[d]_{f_*} & \p (X)  \ar[d]^{f_*} \\
C(Y) \ar[r]_\rho  & \p (Y) 
}
\]
commutes.
\end{coro}

\begin{proof} The corollary follows from the commutativity of the diagram
\[
\xymatrix{
C(x) \ar[r] \ar[d]  & \p (x) \ar[d] \\
 C(f(x)) \ar[r]  & \p(f(x)) 
}
\]
and the last two propositions, where $x\in X$ is an arbitrary closed point. \end{proof}

\begin{prop}\label{Dense}
The image of $\rho$ is dense in $\p(X)$. 
\end{prop}

\begin{proof} Let $V\le \p (X)$ be an open subgroup containing the image of $\rho$. We have to show that $V=\p (X)$.
 If this was not the case $V$ would define a nontrivial abelian Galois covering of $X$ which would split completely over
all closed points of $X$. But this is impossible in view of Proposition~\ref{GlobalSplit}.
\end{proof}

\section{Connected Component}

\noindent In view of the one-dimensional case it is a natural question to ask whether for an arithmetic scheme $X$ the connected component of the identity $C(X)^0$ in the class group $C(X)$ is divisible. In fact in the one-dimensional case it is well known to be divisible and its torsion subgroup can be described at least conjecturally.
The next example shows that $C(X)^0$ is not divisible in general.

\begin{example}\label{examcon}
Let $K$ be a totally imaginary number field of class number $1$ and let $X$ be $\mathbb{A}^1_{\mathcal{O}_K}$. 
By Proposition~\ref{Intersec} we have an exact sequence of topological groups
\[
 \bigoplus \mathbb{C}^\times \lr C(X)^0 \lr C^f(X)^0\lr 0
\]
where the sum on the left is over all archimedean places of curves on $X$. Therefore it suffices to show that $C^f(X)^0$ is not divisible in order
to deduce that $C(X)^0 $ is not divisible.
%Let $R\le I^f(X)$ be the image of the map
%$\bigoplus_{C} \mathbf{k}(C)^\times \to I^f(X)$. Proposition~\ref{Intersec} implies that the natural map
%$\overline{R}/R  \to C^f(X)^0$ is bijective. Let $\iota_x:\mathbb{Z}\to C^f(X)$ be the natural homomorphism corresponding to
%a point $x\in |X|$. 
First, we want to show that $\iota_x$ vanishes for all $x\in |X|$. Fix such an $x$ and let $\mathfrak{p}\subset \mathcal{O}_K$ be
the prime ideal over which $x$ lies, $x$ corresponds to
a monic irreducible polynomial over $\mathcal{O}_K/\mathfrak{p}$. If we lift this polynomial to a monic polynomial over $\mathcal{O}_K$ this lifted
polynomial will be irreducible and therefore gives us a curve  $C$ on $X$ which is finite over $\mathcal{O}_K$. By construction $C$ is \'etale and does not split over $\mathfrak{p}\subset \mathcal{O}_K$. In other words $\mathfrak{p}$ generates the prime ideal in $H^0(\tilde C,\mathcal{O}_{\tilde C})$ corresponding
to the point $x\in \tilde C$, which is therefore principal, since $\mathfrak{p}$ is itself principal ($\mathcal{O}_K$ has class number $1$). Summarizing we
get a commutative diagram
\[
\xymatrix{
\mathbb{Z} \ar[r]^{\iota_x} \ar[dr] &  C^f(X)  \\
 &  \ar[u] C^f(\tilde C)
}
\]
in which the diagonal arrow vanishes. Therefore $\iota_x$ vanishes.  
By Lemma~\ref{Predense} this means that $\{ 0 \}$ is dense in $C^f(X)$ and as $C^f(X)^0$ is a closed subgroup in $C^f(X)$
we deduce that $C^f(X)^0 = C^f(X)$. We will now show  that for any prime number $p$ the group $C^f(X)$ is not $p$-divisible. In fact
$I^f(X)$ contains a direct summand of the form $\mathbb{F}((t))$, where $\mathbb{F}$ is a finite field of characteristic $p$. As 
$\mathbb{F}((t))^\times / (\mathbb{F}((t))^\times)^p $ is uncountable the group $I^f(X)/ p I^f(X)$ is also uncountable. Finally we have a right
exact sequence
\[
\bigoplus_{C} \mathbf{k}(C)^\times  \lr  I^f(X)/p  \lr C^f(X)/p  \lr 0
\]
where the left group is countable, the group in the middle is uncountable, so that the group on the right cannot be trivial.
\end{example}

This example suggests that the only obstruction for $C(X)^0$ to be divisible comes from the local fields of positive characteristic. In fact
if these are absent $C(X)^0$ is divisible.

\begin{theo}\label{ConCom}
For an arithmetic scheme $X$ the connected component $C(X)^0$ in $C(X)$ is divisible if all vertical curves on $X$ are proper. 
\end{theo} 

\begin{rem} Let $X$ be an arithmetic scheme and let $K$ be the algebraic closure of $\mathbb{Q}$ in $\mathbf{k}(X)$.
Recall that all vertical curves on $X$ are proper if and only if the morphism $X\to \mathcal{O}_K$ is proper over its image. This can be seen as
follows: Assume all vertical curves on $X$ are proper and denote the image of $X\to \mathcal{O}_K$ by $U$. Choose a compactification $X\subset \bar{X} \to U$ as constructed by Nagata, for a modern presentation see for example~\cite{Luet}. We have to show $X=\bar X$. For this it suffices to verify 
that for all closed points $u\in U$ we have $X_u=\bar X_u$. Observe that our assumption on the properness of vertical curves implies that $X_u$
is an open and closed subscheme of $\bar{X}_u$.
As $\bar X \otimes K $ is geometrically connected over $K$ we deduce that for all closed $u\in U$ the fibre $\bar X_u$ is geometrically connected
by Zariski's connectedness theorem~\cite[III.11.3]{Hart}.
So we conclude $X_u=\bar X_u$.
\end{rem}

\noindent {\em Proof of Theorem \ref{ConCom}.}
Let $U(X)$ be the open subgroup of $I(X)/I(X)^0$ given by the sum $U(X)=\oplus \mathcal{O}_{K_v}^\times $ over all non-archimedean valuations
appearing in $I(X)$. Let $R(X)$ be the image of 
\[
\bigoplus_C \mathbf{k}(C)^\times \lr I(X) /I(X)^0\; .
\]
Proposition~\ref{Intersec} shows that  we have an exact sequence
\[
I(X)^0 \lr C(X)^0 \lr \overline{R(X)} /R(X)\lr 0\; .
\]
As $I(X)^0$ is clearly divisible it suffices to show that $\overline{R(X)}/R(X)$ is divisible.

\begin{lem}\label{helpcon}
Let $G$ be an abelian topological group, $U\le G$ an open subgroup and $R\le G$ an arbitrary subgroup. Then we have 
\[
\overline{R\cap U} \, \cdot R = \overline{R}\; .
\]
In particular, if $\overline{R\cap U} /R\cap U$ is divisible then $\overline{R}/R$ is divisible, too.
\end{lem}

The proof of the lemma is left to the reader.
Let $C'_i$ ($i\in  \mathbb{N}$) be a family of curves on $X$ containing each curve at least once and let $C_i$ be $\cup_{j\le i} C'_j$. Exactly as in the arithmetic case
we can define the groups $I(C_i), C(C_i), U(C_i)$ and $R(C_i)$.
Our aim is to show:
\begin{enumerate}
\item{$\overline{R(C_i)}/R(C_i)$ is divisible for every $i\in \mathbb{N}$,}
\item{ $\overline{R(X)}=\lim\limits_{\stackrel{\lr}{i}} \overline{R(C_i)} $.}
\end{enumerate}
Here $\overline{R(C_i)}$ means the closure of $R(C_i)$ in $I(C_i)/I(C_i)^0$.
 Indeed the two properties imply that
\[
\overline{R(X)}/R(X) = \lim_{\stackrel{\lr}{i}} \overline{R(C_i)} /R(C_i)
\]
is divisible and this will prove the theorem.
\begin{proof}[Proof of {\rm (1)}.]
It is well known that $U(C_i)$ is a $\hat{\mathbb{Z}}$-module and that $R(C_i)\cap U(C_i)$ is finitely generated as an abelian group. Choosing $d$
generators of this latter group we get a commutative diagram with exact rows and columns 
\[
\xymatrix{
\Z^d \ar[r] \ar[d]&  R(C_i)\cap U(C_i) \ar[d] \ar[r] & 0\\
\hat{\Z}^d  \ar[r]\ar[d]  & \overline{R(C_i)\cap U(C_i)} \ar[r] \ar[d]& 0\\
(\hat{\Z}/\Z)^d \ar[r]\ar[d] &  \overline{R(C_i)\cap U(C_i)} / R(C_i)\cap U(C_i) \ar[r] \ar[d] & 0\\
0 & 0& 
}
\]
and since $\hat{\Z}/\Z$ is divisible we conclude by Lemma~\ref{helpcon} that $\overline{R(C_i)}/R(C_i)$ is divisible.
\end{proof}

\begin{proof}[Proof of {\rm (2)}.]
By Lemma~\ref{helpcon} it suffices to show that 
\[
\overline{R(X)\cap U(X)}=\lim\limits_{\stackrel{\lr}{i}} \overline{R(C_i)\cap U(C_i)  } \; .
\]
Moreover the right hand side is automatically a dense subgroup of the left hand side, so we need to show that the right hand side is
closed in $U(X)$.
For $j>i$ consider the subgroup $S_{i,j}$ of $U(C_i)$ defined to be the preimage of $\overline{R(C_j)\cap U(C_j) }$ under the map $U(C_i)\to U(C_j)$.
$S_{i,j}$ is a closed subgroup of $U(C_i)$ and therefore a $\hat{\Z}$-submodule. Now the essential observation is that $U(C_i)$ is a noetherian
$\hat{\Z}$-module, since it is a finite sum of unit groups of local fields of characteristic $0$; in fact the assumption on the properness of
vertical curves guarantees that no local fields of positive characteristic show up, whose unit group is highly non-noetherian. This implies that for fixed $i$ the ascending sequence of
$\hat{\Z}$-submodules $S_{i,j} \le U(C_i)$ ($j>i$) becomes stationary at some point, in particular $$S_i:=\lim_{\stackrel{\lr}{j}} S_{i,j}$$
is a closed subgroup of $U(C_i)$. Furthermore it is immediate that 
\[
\lim\limits_{\stackrel{\lr}{i}} \overline{R(C_i)\cap U(C_i)  } =   \lim\limits_{\stackrel{\lr}{i}} S_i =: L \; .
\]
As we saw above it is sufficient to show that $L\le U(X)$ is a closed subgroup. By the definition of the direct limit topology $L$ is closed in $U(X)$ if and only if its preimages in the groups $U(C_i)$ are closed for all $i\in \mathbb{N}$. But the preimage of $L$ in $U(C_i)$ is $S_i$, which we
have just seen to be closed. This finishes the proof of property $\mathrm{(2)}$ and therefore of the theorem.
\end{proof}

%-------------------------------

\section{Isomorphism Theorem}\label{SecIso}

\noindent Let $\phi: Y\to X$ be a Galois covering of regular arithmetic schemes.

\begin{theo}[Isomorphism]\label{Iso}
The reciprocity map induces an isomorphism
\[
\rho_{Y/X}:C(X)/ \phi_*C(Y)  \tilde{\lr} \Gal(Y/X)^{ab}\; .
\]
\end{theo}

\begin{proof} Clearly $\Gal(Y/X)^{ab} = \p (X) / \phi_* \p (Y)$ and the surjectivity of $\rho_{Y/X}$ follows from
Proposition \ref{Dense}. Let $U$ be an affine dense open subscheme of $X$ which is smooth over $\mathbb{Z}$.
 Let $W\subset |U|$ be a finite subset and let $$\im_W := \im[\oplus_W \mathbb{Z}
\to  C(X)/\phi_* C(Y)]\; .$$ 
The Bloch approximation method, Proposition \ref{Approx}, produces a curve $D$ 
which contains all points in $W$ as regular points and such that 
$D_{Y}=D\times_X Y $ is irreducible. So we get an isomorphism
\[
\Gal(\mathbf{k}(D_{Y}) /\mathbf{k}(D) ) \tilde{\lr} \Gal(Y/X)
\]
In other words the map $\beta$ in the following commutative diagram is an isomorphism.
\[
\xymatrix{
C(\tilde{D})/ \phi_* C(\tilde{D}_Y)  \ar[r] \ar[d]_\alpha  & \p (\tilde{D}) / \phi_* \p (\tilde{D}_Y)   \ar@{=}[r] \ar[d]^\beta & \Gal(\mathbf{k}(D_Y)/\mathbf{k}(D))^{ab}  \ar[d]^{\wr} \\
 C(X)/ \phi_* C(Y) \ar[r]^{\rho_{Y/X}}  &  \p (X) /\phi_* \p (Y)  \ar@{=}[r] & \Gal (Y/X)^{ab} }
\]
Here $\tilde{D}$ is the normalization of $D$. 
According to Corollary \ref{IsoOne} the upper horizontal arrow is an isomorphism, so that we deduce 
$\im(\alpha) \cap \ker (\rho_{Y/X}) = 0 $. 
Furthermore it is clear by the choice of $D$
that $\im_W \subset  \im(\alpha)$ and consequently $\im_W \cap \ker(\rho_{Y/X}) = 0 $. As this holds for
all finite subsets $W\subset U$ and we know according to Lemma~\ref{Predense} that 
$$\bigcup_W \im_W = C(X)/\phi_* C(Y)$$ we deduce $\ker(\rho_{Y/X}) = 0$. Here we use the fact that $\phi_* C(Y)$
is an open subgroup of $C(X)$, which is true since for a finite extension of local field the norm map on the multiplicative groups of the fields is open.
 This finishes the proof of the theorem.\end{proof}

%---------------------------------
%---------------------------------

\section{Weak Existence Theorem}

\subsection{Extension}\mbox{}\\
Consider the commutative diagram
\[
\xymatrix{
C(X') \ar[r]^\rho \ar[d]_{i_*}  &   \p (X') \ar[d]^{i_*} \\
C(X) \ar[r]_\rho  &  \p (X)
}
\]
where $X,X'$ are regular nonempty arithmetic schemes and $i:X'\to X$ is an open immersion.

\begin{prop}[Extension]\label{Extension}
Let $V'\le \p (X')$ and $U\le C(X)$ be open subgroups such that $\rho^{-1} (V') = i_*^{-1} (U)$. Then there exists a unique open subgroup
$V\le \p (X)$ with $\rho^{-1} (V) = U$ and $i_*^{-1}(V) =V'$.
\end{prop}

\begin{proof} Uniqueness is clear because $\p(X')\to \p(X)$ is surjective. For existence let 
$V'\le \p (X')$ corresponds to the abelian Galois covering $\phi':Y'\to X'$. Now let $\phi:Y\to X$ be the normalization
of $X$ in $\mathbf{k}(Y')$. We have to show $\phi $ is \'etale, because then $\phi$ corresponds to some open subgroup $V\le \p(X)$
with $i_*^{-1}(V)=V'$. It follows easily from the fact that the image of $C(X')\to C(X)$ is dense, Lemma~\ref{Predense}, 
that we also have $\rho^{-1}(V)=U$. So it suffices to show
the following claim.
 Fix a closed point $x\in |X|-|X'|$ and let $X_x=Spec(A_x)$ be an \'etale extension of the henselian local ring at $x$ such that pullback
of $U$ to the residue field of $A_x$ is trivial. Let $X'_x$ be $X'\times_X X_x$. \\[2mm]
{\bf Claim.}
The \'etale covering $Y'\times_X X'_x \to X'_x$ splits completely over all closed points.
\begin{proof}
$Y\times_X X_x$ is normal. It is \'etale over $X'_x$ and each of
its components is an abelian Galois covering over $X'_x$. The closed points
of $X'_x$ correspond to the branches of curves on $X$ through $x$ which meet $X'$.\\
For such a curve $D$ on $X$ with normalization $\tilde D$ denote by $j:\tilde D \to X$ the natural map and by $j':\tilde D' \to X'$ its
restriction to $X'$.
The pullback $j'^{-1} ( Y)$ is an \'etale covering of $D'$ which corresponds by one-dimensional
class field theory, Proposition \ref{OneClass}, to the open subgroup $(i \circ j')_*^{-1}(U)\le C(\tilde{D}')$.
But then one-dimensional class field theory
 implies that $j'^{-1} ( Y)$ extends to an \'etale covering of $\tilde{D}$. \\
Now we claim that this implies that $Y \times X_x$ splits completely over all closed points of $X'_x$. 

Let $y \in X'_x$ be such a closed point corresponding to the prime ideal $p_y \subset A_x$. By what has been said above $Y \times \mathbf{k} (y)$ extends
to an \'etale covering of the normalization of $A_x/p_y$. 
But over the closed point of the normalization of $A_x/p_y$ this \'etale covering splits completely, since the 
pullback of $U$ to the residue field of $A_x$ is trivial. As the normalization of $A_x/p_y$ is a henselian discrete 
valuation ring, the whole \'etale covering is the sum of trivial coverings.
\end{proof}
The claim together with the local splitting result, Proposition \ref{LocalSplit}, implies that $Y\times_X X_x\to X_x$ is
a sum of trivial coverings and thus that $Y$ is \'etale over $X$.\end{proof}

\subsection{Effacability}\mbox{}\\
The next proposition is one of the key observations of Wiesend's approach to the class field theory of arithmetic schemes. It says that \'etale locally around the generic point of an arithmetic scheme
$X$ an open subgroup $U\le C(X)$ becomes trivial.

\begin{prop}\label{Eff}
Let $U\le C(X)$ be an open subgroup such that $C(X) /U$ has finite exponent. 
Then there exists an arithmetic scheme $X'$ and an \'etale morphism $f:X'\to X$ with dense image such that $f_*^{-1}(U)= C(X')$.
\end{prop}

\begin{proof} Let $m$ be the exponent of $C(X)/U$ and assume without restriction that $m\in H^0(X,\mathcal{O}_X^\times)$. 
We use induction on $d=\dim(X)$. If $d=1$ the proposition can be deduced from one-dimensional class field theory, Proposition
\ref{OneClass}. So assume $d>1$. Using Proposition \ref{ElFib} we can without restriction assume that
there exists an open immersion $i:X\to \bar{X}$ and a relative curve $\bar{\pi}:\bar{X} \to W$ over a regular
arithmetic scheme $W$ such that $\bar{X} -X$ is a direct sum of copies of $W$. Furthermore we can assume that there
exists a section $s:W\to X$. Set $\pi = \bar{\pi} \circ i$.

By induction $s_*^{-1}(U)$ becomes trivial in some \'etale neighborhood of the generic point of $W$. So after
an \'etale base change we can assume that $s_*^{-1}(U)=C(W)$. Furthermore replacing $W$ by an \'etale neighborhood of its
generic point 
we can assume that $R^1 \pi_* ( \mathbb{Z} /m)$ is constant on $W$ using Proposition \ref{EFiniteness}.
Let $Y_{\mathbf{k}(W)}$ be the maximal abelian Galois covering of $X_{\mathbf{k}(W)}=X\times_W \mathbf{k}(W)$
of exponent $m$ which splits completely over $s_{\mathbf{k}(W)}:Spec(\mathbf{k}(W)) \to X_{\mathbf{k}(W)}$. As
a first step we want to determine the Galois group of $Y_{\mathbf{k}(W)}  \to X_{\mathbf{k}(W)}$. We have an
isomorphism
\[
\p (X_{\mathbf{k}(W)}) \cong \p (X_{\overline{\mathbf{k}(W)}})_{\Gal(\mathbf{k}(W))} \oplus \Gal(\mathbf{k}(W))^{ab}
\] 
induced by $s$  and an isomorphism
\[
Hom_{cont}(\p (X_{\overline{\mathbf{k}(W)}} ), \mathbb{Z}/m ) = H^1 (X_{\overline{\mathbf{k}(W)}} ,\mathbb{Z}/m) \; ,
\]
which makes $\p (X_{\overline{\mathbf{k}(W)}} ) \otimes \mathbb{Z}/m$ and $ H^1 (X_{\overline{\mathbf{k}(W)}} ,\mathbb{Z}/m)$ dual finite abelian groups.
Base change, Proposition \ref{EFiniteness}, and the assumption that $R^1 \pi_* ( \mathbb{Z} /m)$ is a constant sheaf 
imply that $\Gal(\mathbf{k}(W))$ acts trivially on
$H^1 (X_{\overline{\mathbf{k}(W)}} ,\mathbb{Z}/m)$ and therefore on $\p (X_{\overline{\mathbf{k}(W)}})\otimes \mathbb{Z}/m$. 
So we have 
\[
\p (X_{\mathbf{k}(W)}) \otimes \mathbb{Z}/m \cong \p (X_{\overline{\mathbf{k}(W)}}) \otimes \mathbb{Z}/m \oplus \Gal(\mathbf{k}(W))^{ab} \otimes \mathbb{Z}/m
\]
and $$\Gal(  Y_{\mathbf{k}(W)}/X_{\mathbf{k}(W)}) = 
 \p ( X_{\overline{\mathbf{k}(W)} }) \otimes \mathbb{Z}/m \; .$$
Let $\phi:Y\to X$ be the normalization of $X$ in $\mathbf{k}(Y_{\mathbf{k}(W)})$. \\[2mm]
{\bf Claim.} {$\phi:Y\to X$ is an \'etale morphism and induces over each fibre of $\pi:X\to W$ the maximal abelian 
Galois covering of exponent $m$ completely split over the image of $s$.}

\begin{proof} Let $w$ be a point of $W$ and let $A^h$ resp.~$A^{sh}$ the henselization resp.~strict henselization of the 
local ring at $w$, $F^h=Q(A^h)$. 
Let $Y_{A^h}$ be the maximal abelian Galois covering of $X_{A^h}=X\times A^h$ of exponent $m$ completely split over the image of $s_{A^h}$.
We will show that $Y_{A^h}$ is indeed isomorphic to the base change $Y\times A^h$, because this implies that $Y\to X$
is \'etale over $w$.
Reasoning as above we get 
\[
\Gal(Y_{A^h}/X_{A^h}) = \p (X_{A^{sh}} )  \otimes \mathbb{Z}/m 
\] 
So the equalities
\[
\p (X_{\overline{\mathbf{k}(W) }}) \otimes \mathbb{Z}/m = \p (X_{A^{sh}} ) \otimes \mathbb{Z}/m = \p (X_{\overline{\mathbf{k}(w)} }) \otimes \mathbb{Z}/m \; ,
\]
which follow from Proposition \ref{EFiniteness},
show that $Y_{\mathbf{k}(W)} \times_{\mathbf{k}(W)} F^h = Y_{A^h} \times_{A^h} F^h $ and that 
$Y_{A^h}\times_{A^h} \mathbf{k}(w)$ is the maximal abelian Galois covering of exponent $m$ of $X_w$ completely split 
over the image of $s_w$. 
The former implies that $Y_{A^h }=Y\times A^h$ as we wanted to show and the latter shows that the fiber of $Y\to X$ over $w$ is
what it should be. 
\end{proof}

Finally, for a closed point $w\in |W|$ consider the commutative diagram
\[
\xymatrix{
Y_w \ar[r] \ar[d]_{\phi_w}  &  Y \ar[d]^{\phi} \\
X_w \ar[r]_{i_w}  &  X 
}
\]
One-dimensional class field theory, Proposition \ref{OneClass},  shows that $(i_w \circ \phi_w  )_*^{-1} (U) =C(Y_w)$
because $\phi_w:Y_w\to X_w$ is the maximal abelian Galois extension of exponent $m$ completely split over the image
of $s$. The image of
\[
\bigoplus_{w\in |W|} C(Y_w) \lr  C(Y)  
\]
is dense, Lemma~\ref{Predense}, and $\phi_*^{-1} (U)\le C(Y)$ is an open subgroup and contains this image. 
So it follows that $\phi_*^{-1}(U)=C(Y)$. We set $X'=Y$. \end{proof}

\subsection{Existence}\mbox{}\\
Let $X$ be a regular arithmetic scheme.

\begin{theo}[Weak Existence]\label{WeakEx}
Let $U\le C(X)$ be an open subgroup such that $C(X) /U$ has finite exponent. Then there exists a unique open subgroup $V\le \p (X)$ with
$\rho^{-1} (V)=U$. In particular $C(X)/U$ is finite.
\end{theo}

In fact the deep finiteness theorem of the next section will show that the finite exponent assumption is superfluous. This is why we call
the theorem weak existence theorem. In the proof of the finiteness theorem one uses the weak existence theorem in an essential way, so that it deserves
its own name.
\begin{proof} The uniqueness holds because $\rho$ has dense image according to Proposition~\ref{Dense}. 
In fact this implies that for an open subgroup $V\le \p(X) $ the subgroup $\rho(\rho^{-1}(V)) $ is dense in $V$, therefore its closure is $V$. 
So if $V_1, V_2 \le \p (X)$ both have preimage $U$ we have the equalities  
\[
V_1 = \overline{\rho (U)} = V_2 \; .
\]
For the existence part use Theorem \ref{Eff} to find a nontrivial \'etale morphism $f:X''\to X$ with $f_*^{-1} (U) = C(X'')$.
From Lemma \ref{fcov} we know that after replacing $X''$ by an \'etale neighborhood of its generic point we can factor $f$ into a
Galois covering $\phi: X'' \to X'$ and an open immersion $i:X' \to X$. It follows from our assumption
that $\phi_* C(X'') \subset i_*^{-1}(U) = U'$. The isomorphism theorem, Theorem \ref{Iso}, gives an isomorphism
\[
\rho:C(X')/ \phi_* C(X'')  \tilde{\lr} \Gal(X''/X')^{ab}= \p (X') / \phi_* \p (X'') \; .
\]
so that $V'= \rho (U') + \phi_* \p (X'')\le \p (X')$ satisfies $\rho^{-1}(V') = U'$.
Finally, Proposition \ref{Extension} applied to $i:X' \to X$  produces an open subgroups $V\le \p (X)$ with
$\rho^{-1}(V) = U$. \end{proof}

\section{Finiteness}

\noindent Wiesend's finiteness theorem is one of the strongest and most beautiful results in higher global class field theory. In some sense it replaces 
Bloch's exact sequence in the more classical approaches to higher global class field theory originating from~\cite{Bloch}. The proof we give is a corrected version
of Wiesend's proof in \cite{W4} close to~\cite[Section 5]{KeSch}. Let $X$ be a
regular arithmetic scheme.

\begin{theo}[Finiteness]\label{Finiteness}
Let $U\le C(X)$ be an open subgroup. Then $U$ has finite index in $C(X)$.
\end{theo}

\begin{proof} By the weak existence theorem, Theorem \ref{WeakEx}, it is enough to show that $C(X)/U$ has finite
exponent. But in order to prove the latter it is sufficient to show that for some \'etale neighborhood $f:X'\to X$ of the generic
point of  $X$ the exponent of $C(X')/f_*^{-1}(U)$ is finite. In fact this follows from the short exact sequence
\[
0\lr C(X')/f_*^{-1}(U) \lr C(X) / U \lr C(X)/ [f_*(C(X')) + U] \lr 0
\]
because $\mathrm{exp} [C(X) / f_*(C(X')) + U]\le \deg(X'/X)$. We will use this reduction several times in the proof.
In particular we can assume that there
exists a fibration $\pi:X\subset \bar{X} \to W$ as in Proposition \ref{ElFib} with $W$ affine and smooth over $\mathbb{Z}$.
 Furthermore we can by induction and using the weak existence
theorem, Theorem \ref{WeakEx}, assume that
the pullback of $U$ along some section $s:W\to X$ of $\pi$ is trivial.
For the rest of the proof we let $N_x$ be the order of the image of $C(x)\to C(X)/U$ for $x\in |X|$.
$N_x$ is finite for all $x\in |X|$, since for a horizontal curve $i:D\to X$ with $x\in D$ regular
we have $N_x \le \# C(D)/i_*^{-1}(U)< \infty$. The latter finiteness holds by one-dimensional class field theory, Proposition
\ref{OneClass}.
 By Lemma \ref{Predense} it is enough
to show that all the $N_x$ are bounded. First we prove the theorem under the following assumption: \\[2mm]
{\bf Assumption.} There is a prime
$l$ such that for all $x\in |X|$ the natural number $N_x$ is a power of $l$.  \\[2mm]
Without restriction we can assume that $l\in H^0(W,\mathcal{O}_W^\times)$. Then by Proposition  \ref{EFiniteness}
we get that 
$R^1 \pi_* \mathbb{Q}_l /\mathbb{Z}_l$ is locally constant. It is easy to see that the $l^i$-torsion
subsheaf of $R^1 \pi_* \mathbb{Q}_l/\mathbb{Z}_l$ is $R^1 \pi_* \mathbb{Z}/l^{i}$ for all $i>0$. The next claim assures the existence
of an $l$-Bloch point.  \\[2mm]
{\bf Claim.} There exists an arithmetic scheme $W'$, an \'etale morphism $W'\to W$ with dense image and a closed point $w_0\in |W'|$ such that 
\[
H^1(X_{\overline{\mathbf{k}(W')}},\mathbb{Q}_l / \mathbb{Z}_l )^{\Gal(\mathbf{k}(W'))} \cong  H^1(X_{\overline{\mathbf{k}(w_0)}},\mathbb{Q}_l / \mathbb{Z}_l )^{\Gal(\mathbf{k}(w_0))}\; .
\]

\begin{proof} Remark that by our assumptions the left hand side is contained in the right hand side for arbitrary $w_0$ and $W'$. In order to 
see this observe that by Proposition \ref{EFiniteness} we have
\[
H^1(X_{\overline{\mathbf{k}(w_0)}},\mathbb{Q}_l / \mathbb{Z}_l ) = H^1(X_{A^{sh}},\mathbb{Q}_l / \mathbb{Z}_l )=
H^1(X_{\overline{\mathbf{k}(W')}}, \mathbb{Q}_l /\mathbb{Z}_l )\; .
\]
Here $A^h$ resp.~$A^{sh}$ is the henselization resp.~strict henselization of the local ring at some $w_0\in W'$. 
Furthermore we have 
$\Gal(\mathbf{k}(w_0)) =\pi_1 (Spec( A^h)) \subset \Gal (\mathbf{k}(W'))$, so putting this together we get
\begin{eqnarray*}
H^1(X_{\overline{\mathbf{k}(w_0)}},\mathbb{Q}_l / \mathbb{Z}_l )^{\Gal(\mathbf{k}(w_0))} & = & H^1(X_{A^{sh}},\mathbb{Q}_l / \mathbb{Z}_l )^{\pi_1 (Spec(A^h))}\\
& \supset & H^1(X_{\overline{\mathbf{k}(W')}},\mathbb{Q}_l / \mathbb{Z}_l )^{\Gal(\mathbf{k}(W'))} \; .
\end{eqnarray*}

Observing that these groups  are finite by Theorem \ref{Katz-Lang} we fix $w\in |W|$ and let $W'$ be a sufficiently fine 
\'etale neighborhood of $w$ such that there exists $w_0\in |W'|$ with $\mathbf{k}(w_0) =\mathbf{k}(w)$ and such that the last inclusion becomes an equality. 
The details can be found in \cite[Proposition 5.7]{KS}.\end{proof}
For the rest of the proof of the theorem we assume without restriction that there exists an $l$-Bloch point $w_0\in W$, 
that is a point as in the last claim. The generic point of $W$ is denoted by $\eta$. We introduce the following notation
\[
H(w,l^i)= \# H^1(X_{\overline{\mathbf{k}(w)}},\mathbb{Z}/ l^i )^{\Gal(\mathbf{k}(w))}
\] for $w\in W$ and $i\in \mathbb{N} \cup \{\infty\}$; here $\mathbb{Z}/l^\infty = \mathbb{Q}_l / \mathbb{Z}_l$. It follows from Katz-Lang finiteness,
Theorem \ref{Katz-Lang}, that for all points $w\in W$ we have $H(w,l^\infty)< \infty$.  \\[2mm]
{\bf Claim.} There exists a Galois covering $W'\to W$  such that
$H(w,l^\infty)=H(w_0,l^\infty)$ for all points $w\in |W|$ with $\DG_{W'/W}(w_0)\subset \DG_{W'/W}(w)$.
\begin{proof} Choose $n>0$ such that  $H(w_0,l^\infty) = H(w_0, l^n)$ and let $W'$ be the Galois covering trivializing $R^1 \pi_* \mathbb{Z}/ l^{n+1}$. It is then sufficient to show
that $H(w_0,l^n)= H(w,l^n)=H(w,l^{n+1})$ for a point $w\in |W|$ with  $\DG_{W'/W}(w_0)\subset \DG_{W'/W}(w)$, since if for a finite $l$-primary abelian group
exponent $l^n$ resp. $l^{n+1}$ elements coincide the group has exponent $l^n$ itself. So fix a point $w$ with 
$\DG_{W'/W}(w_0)\subset \DG_{W'/W}(w)$ and observe that this implies $H(w,l^i)\le H(w_0 ,l^i)$ for $i\le n+1$. Generally we have for 
every point $w\in W$ and $i\in \mathbb{N}\cup \{\infty\} $ the inequality $H(\eta, l^i) \le H(w,l^i)$. On the other hand the assumption that
$w_0$ is an $l$-Bloch point means $H(w_0 , l^i) = H(\eta , l^i)$. If we put everything together we see $H(w_0,l^n) = H(w,l^n)=H(w,l^{n+1})$. 
\end{proof}

Set $X'=X \times_W W'$ and $x_0=s(w_0)$. It is easily seen that $X'\to X$ is a Galois covering with Galois group $\Gal(W'/W)$.\\[2mm]
{\bf Claim.}  The  numbers $N_x$ are bounded above if $x\in |X|$ varies over all points with $\DG_{X'/X}(x) = \DG_{X'/X}(x_0)$.
\begin{proof} Using one-dimensional class field theory, Proposition \ref{OneClass},  it follows that 
$N_x\le \# C(X_{\pi(x)})/U_{\pi(x)}\le H(\pi(x) , l^\infty)$, because the pullback of $U\le C(X)$ along $s:W\to X$ is trivial. Here $U_{\pi(x)}$
is the inverse image of $U$ in $C(X_{\pi(x)}) $.
Now if $\DG_{X'/X}(x) = \DG_{X'/X}(x_0)$ we have $$\DG_{W'/W}(w_0) = \DG_{X'/X}(x_0) = \DG_{X'/X}(x) \subset \DG_{W'/W}(\pi(x))\; ,$$ 
so that by the last claim $H(\pi(x),l^\infty) = H(w_0, l^\infty)$. This means the $N_x$ in question 
are bounded by $H(w_0,l^\infty )$. \end{proof}

Let $B$ be a bound for the numbers $N_x$ with $\DG_{X'/X}(x)=\DG_{X'/X}(x_0)$ and $d=\deg(X'/X)$.\\[2mm]
{\bf Claim.} $N_x\le B\, d$ for all $x\in |X|$.
\begin{proof} Let $x\in |X|$. We will show $N_x\le B\, d$. Let $C\hookrightarrow X$ be a curve which contains $x_0$ and $x$ as regular points and such that
$C'=C\times_X X'$ is irreducible. Such a curve exists according to the Bloch approximation method, Proposition \ref{Approx}. Set
\[
A = \{ y\in |C| \, |\, y \text{ regular}, \DG_{C'/C}(y)=\DG_{C'/C}(x_0) \}\; .
\]
Chebotarev density, Proposition \ref{Chebotarev}, shows that $D(A)\ge 1/d$. Set
\[
A' = \{ y\in |C| \, | \, y \text{ regular}, N_y\le B \}\; .
\]
By the last claim $A\subset A'$ and consequently $D(A')\ge D(A) \ge 1/d$. One-dimensional class field theory, Theorem \ref{OneClass},
gives a canonical abelian Galois covering  $C_U\to \tilde C$ whose Galois group $G$ is naturally isomorphic to a subgroup of $C(X)/U$.
Then $$A'=\{y\in |C| \, | \, y \text{ regular}, Frob_y\in  _B\!\!G  \}\; ,$$ where $Frob_y$ is the Frobenius in $G$
and $ _B G$ is the $B$-torsion subgroup of $G$. 
 Assume for the moment that $N_x> B\, d$; this will lead to a contradiction.
It would imply $\# G/ _B G > d$. But then the Chebotarev density theorem applied to $C_U \to \tilde C$ would give
$D(A')<1/d$, which is a contradiction to $D(A')\ge 1/d$.
\end{proof}

{\bf General case.}
The finiteness theorem holds without the assumption that the $N_x$ are powers of some prime.\\[2mm]
Remember that we have reduced the theorem to the case of a fibration $\pi: X\to W$ and a section $s:W\to X$ such that $s_*^{-1}(U)=C(W)$.
For a prime $l$ let $U_l\le C(X)$ be the smallest subgroup containing $U$ such that $C(X)/U_l$ is $l$-primary. 
By the special case treated above $C(X)/U_l$ is finite. 
As $C(X)/U$ is torsion we have
\[
C(X)/U = \oplus_l C(X)/U_l\; ,
\]
so that is suffices to show that for almost all primes $l$ the group $C(X)/U_l$  is trivial. But
Theorem \ref{WeakEx} shows that $C(X)/U_l$ is isomorphic to a quotient of $\ker[ \p(X_{\mathbf{k}(W)}) \to \p( \mathbf{k}(W))]$ and this group
is finite according to Proposition \ref{Katz-Lang}. This finishes the proof of the finiteness theorem.\end{proof}
\begin{coro}
The Chow group of zero cycles $CH_0(X)$ is finite.
\end{coro}

\begin{proof} There exists a natural surjective homomorphism $C(X) \to CH_0(X)$ with open kernel. \end{proof}

\begin{rem}
One can develop an analogous class field theory of smooth varieties over finite fields, see~\cite{Wclass} and~\cite{KeSch}, and can prove the finiteness of the tame class group
in this context by very similar methods. Together with the above finiteness theorem this reproves a result of Bloch and Kato-Saito which says that $CH_0(X)$ is finitely generated 
if $X$ is a scheme of finite type over $\mathbb{Z}$.
\end{rem}

\section{Fundamental theorems}

\noindent The following theorems comprise the essential features of higher class field theory of arithmetic schemes.

\begin{theo}\label{Main}
For a regular arithmetic scheme $X$ the sequence
\[
0\lr C(X)^0 \lr C(X) \stackrel{\rho}{\lr} \p (X) \lr 0
\]
is a topological short exact sequence. Here $C(X)^0$ is the connected component in $C(X)$.
\end{theo}

\begin{theo}\label{Main2}
If $X$ is an arithmetic scheme the connected component of the class group $C(X)^0$ has the following characterizations:
\begin{enumerate}
\item[\rm (i)]{ It is the intersection of all open subgroups of $C(X)$. }
\item[\rm (ii)]{ It is the closure of the image of $I(X)^0\to C(X)$.}
\end{enumerate}
If $X$ is regular we have:
\begin{enumerate}
\item[\rm (iii)]{ It is the group of universal norms in $C(X)$.}
\end{enumerate}
If $X$ is regular and furthermore all vertical curves on $X$ are proper:
\begin{enumerate}
\item[\rm (iv)]{  It is the set of divisible elements of $C(X)$.}
\item[\rm (v)]{ It is the maximal divisible subgroup of $C(X)$.}
\end{enumerate}
\end{theo}

\noindent {\em Proof of Theorem~\ref{Main}.} The weak existence theorem, Theorem~\ref{WeakEx}, together with the finiteness theorem, Theorem \ref{Finiteness}, show that every
open subgroup of $C(X)$ is the preimage of an open subgroup in $\p (X)$. So the exactness at $C(X)$ is clear, 
because the intersection of all open subgroups
of $C(X)$ is $C(X)^0$ by  Proposition \ref{Intersec}. Moreover the global splitting
result, Proposition \ref{GlobalSplit}, shows that $\rho$ has dense image. Now in order to finish the proof we
show that $C(X)/C(X)^0$ is compact.\\[2mm]
{\bf Claim.} { If $X$ maps smoothly and surjectively onto some open subscheme $U \subset Spec(\mathcal{O}_K)$
for some number field $K$ and there exists
a section $s: U \to X$ then $C(X)/C(X)^0$ is compact.}\\[2mm]
We reduce the general case to the one treated in the claim. In fact there exists $X''$ satisfying the
assumption of the claim and an \'etale covering $\phi:X''\to X'$ where $i:X' \to X$ is some open subscheme of $X$.
Let $G$ be defined by the exact sequence
\[
C(X'')/C(X'')^0 \lr C(X')/C(X')^0 \lr G \lr 0\; .
\]
As the image of the first arrow in the sequence is open Theorem \ref{Finiteness} implies that $G$ is finite. As the
claim says that
$C(X'')/C(X'')^0$ is compact this implis that $C(X')/C(X')^0$ is compact too. From Lemma~\ref{Predense}
we known that the image of $C(X')/C(X')^0 \to C(X) /C(X)^0$ is dense and so $C(X)/C(X)^0$ is compact.

\begin{proof}[Proof of the claim.] We are in a situation where we can apply Katz-Lang finiteness, Proposition \ref{Katz-Lang}.
So consider the commutative diagram with exact rows
\[
\xymatrix{
0 \ar[r]  &  ker_1 \ar[r] \ar[d]  &  C(X)/C(X)^0  \ar[r] \ar[d]_\alpha & C(U)/C(U)^0 \ar[d]^\beta \ar[r] & 0 \\
0 \ar[r]  &  ker_2 \ar[r]  &  \p (X)  \ar[r]  & \p (U) \ar[r] & 0 
}
\]

Katz-Lang says that $ker_2$ is finite. The injectivity of $\alpha$ was shown above so that $ker_1$ is finite too.
Finally $\beta$ is an isomorhism by Proposition \ref{OneClass}, so that $C(U)/C(U)^0 $ is compact. The section $s:U\to X$ induces an isomorphism
$C(X)/C(X)^0 \cong ker_1 \oplus C(U)/C(U)^0 $ which completes the proof of the claim. \end{proof}

\begin{proof}[Proof of Theorem~\ref{Main2}]
(i) and (ii) are just Proposition~\ref{Intersec}. (iii) follows immediately from Theorem~\ref{Main}. By Theorem~\ref{ConCom}
we know that $C(X)^0$ is divisible if all vertical curves on $X$ are proper, so for (iv) and (v) we have to show that an element $\alpha \in C(X) \backslash C(X)^0 $ cannot
be divisible. In fact the image of $\alpha$ in $\p(X)$ is non-zero by Theorem~\ref{Main} and a non-zero element of a pro-finite group is not divisible.
\end{proof}

\section{Comparison with Kato-Saito class field theory}

\noindent We shortly recall Kato-Saito class field theory of arithmetic schemes and explain why Wiesend's class field
theory is stronger, i.e.~implies the main results of Kato and Saito. First of all we have to translate Wiesend's theory
from the complete world to the henselian world. For an arithmetic scheme $X$ we define $C^h(X)$ in the same
way as $C(X)$ but replacing the complete field $K_v$ by the henselian local field $K_v^h$, which is defined
as the algebraic closure of $K$ in $K_v$ if $v$ is archimedean. 
\begin{lem}\label{LemAppendix}
The natural map $\phi:C^h(X)\to C(X)$ induces a bijection between the open subgroups of $C(X)$ and the open subgroups of
$C^h(X)$.
\end{lem}

\begin{proof}
Since $\phi$ is continuous and has dense image the natural map
\[
 \{\text{open subgroups of } C(X) \}  \lr  \{\text{open subgroups of } C^h(X)  \}
\]
is injective. So we have to show it is surjective. Let $U^h\le C^h(X)$ be an open subgroup. For every curve $C$ on $X$ and
every $v\in C_\infty^{\mathrm{na}}$ choose $n_v\in \mathbb{N}$ such that the image of $1+\pi_v^{n_v} \mathcal{O}_{\mathbf{k}(C)^h_v}$ in $C^h(X)$
is contained in $U^h$. Here $\pi_v$ is a prime element. Setting 
\[
U= \phi(U^h)+ \im (I(X)^0) + \sum_{v} \im(1+\pi_v^{n_v} \mathcal{O}_{\mathbf{k}(C)_v})
\]
we obtain $\phi^{-1} (U)= U^h$.
\end{proof}

Let $\bar{X}$ be an arithmetic scheme of dimension $d$ which is proper over $\mathbb{Z}$ and let $X\subset \bar{X}$ be a 
dense open subscheme which is smooth over $\mathbb{Z}$. For simplicity we will assume that $\mathbf{k}(X)$ 
contains a totally imaginary number field.\\
For a coherent ideal sheaf $I$ an $\bar{X}$ such that $I|_{X}= \mathcal{O}_X$ 
Kato and Saito define their cohomological class group as
\[
C_I(\bar{X})  =  H^d_{Nis}(\bar{X},\mathcal{K}^M_d(\mathcal{O}_{\bar{X}} , I ))\; .
\]
Here  $\mathcal{K}^M_d(\mathcal{O}_X, I)$ is the relative Milnor $K$-sheaf in the Nisnevich topology, defined as
\[
\mathcal{K}^M_d(\mathcal{O}_{\bar{X}}, I) =  \ker[\mathcal{K}^M_d(\mathcal{O}_{\bar{X}}) \to \mathcal{K}^M_d(\mathcal{O}_{\bar{X}} /I) ]\; .
\]
For general properties of the Milnor $K$-sheaf see \cite{Kerzklocal}.
The main result of Kato and Saito in \cite{KS} and \cite{Raskind} reads now:
\begin{theo}\label{KatoSaito}
For all $I$ as above the group $C_I(\bar{X})$ is finite and there is a natural
reciprocity isomorphism of topological groups
\[
KS(X)=\lim_{\stackrel{\longleftarrow}{I}} C_I(\bar{X} ) \tilde{\lr} \p (X)\; .
\] 
\end{theo}

\begin{proof}
It follows from \cite[Proposition 2.9]{KS} that there is a natural continuous surjective homomorphism $ C^h(X) \to C_I(\bar{X})$
(the Kato-Saito class group $C_I(\bar{X})$ has the discrete topology).
The finiteness theorem, Theorem \ref{Finiteness}, and Lemma \ref{LemAppendix} imply the finiteness of $C_I(\bar{X})$ 
and therefore the pro-finiteness of $KS(X)$. This shows that the continuous homomorphism $KS(X) \to \p (X)$ 
defined by Kato and Saito is surjective, since we already know that the image is dense.
 The diagram
\[
\xymatrix{
C^h(X) \ar[rr]\ar[dr]_\rho & & KS(X) \ar[dl] \\
& \p (X) & 
}
\] 
commutes. 
 Finally, from this diagram, Lemma 
\ref{LemAppendix} and Theorem \ref{Main} it follows that the open subgroups of $KS(X)$ are in bijective correspondence 
with the open subgroups of $\p (X)$, so that the injectivity of $KS(X) \to \p (X)$ results. \end{proof}

\vspace{1cm}

\noindent Moritz Kerz\\
Universit\"at Duisburg-Essen\\
Fachbereich Mathematik, Campus Essen\\
45117 Essen\\
Germany

\noindent{\tt moritz.kerz@uni-due.de}

\end{document}